\documentclass{llncs}
\usepackage{graphicx}
\usepackage{amsmath,amssymb}
\usepackage{mathtools}
\usepackage{algorithm}
\usepackage{algpseudocode}

\usepackage[colorlinks=true, allcolors=blue, linktocpage=true]{hyperref}
\usepackage{enumitem}

\usepackage{xcolor}
\usepackage[normalem]{ulem}
\usepackage{booktabs}

\def\R{\mathbb{R}}

\def\St{\mathrm{St}}

\def\so{\mathfrak{so}}
\def\O{\mathrm{O}}

\def\D{\mathrm{d}}

\def\I{\mathrm{I}}

\def\v{\mathfrak{v}}
\def\h{\mathfrak{h}}

\def\RS{\mathcal{RS}}

\DeclareMathOperator{\Exp}{Exp}

\DeclareMathOperator{\expm}{exp_m}
\DeclareMathOperator{\logm}{log_m}

\DeclareMathOperator{\tr}{tr}

\DeclarePairedDelimiterX{\norm}[1]{\lVert}{\rVert}{#1}

\newcount\Comments  
\newcommand{\kibitz}[2]{\ifnum\Comments=1\textcolor{#1}{#2}\fi}

\begin{document}
\title{Efficient Quasi-Geodesics on the Stiefel Manifold}
%
%
\author{Thomas Bendokat* \and Ralf Zimmermann}
\authorrunning{T. Bendokat \and R. Zimmermann}
%
\institute{Department of Mathematics and Computer Science, University of Southern Denmark~(SDU), Odense, Denmark\\
\email{\{bendokat,zimmermann\}@imada.sdu.dk}}
\maketitle              
\begin{abstract}
%
Solving the so-called geodesic endpoint problem, i.e., finding a geodesic that connects two given points on a manifold, is at the basis of virtually all data processing operations, including averaging, clustering, interpolation and optimization. 
On the Stiefel manifold of orthonormal frames, this problem is computationally involved.
A remedy is to use quasi-geodesics as a replacement for the Riemannian geodesics.
Quasi-geodesics feature constant speed and covariant acceleration with constant (but possibly non-zero) norm. For a well-known type of quasi-geodesics, we derive a new representation that is suited for large-scale computations. Moreover, we introduce a new kind of quasi-geodesics that turns out to be much closer to the Riemannian geodesics.

\end{abstract}
%
%
%
\section{Introduction}
Connecting two points on the Stiefel manifold with a geodesic requires the use of an iterative algorithm \cite{Zimmermann2017}, which raises issues such as convergence and computational costs. An alternative is to use \emph{quasi-geodesics} \cite{BatistaKrakowskiLeite2017,JurdjevicMarkinaLeite2019,KrakowskiMachadoLeiteBatista2017,MachadoLeiteBatzies2020,NishimoriAkaho2005}. The term is used inconsistently. Here, we mean {\em curves with constant speed and covariant acceleration with constant (but possibly non-zero) norm}. 
The term quasi-geodesics is motivated by the fact that actual geodesics feature a constant-zero covariant acceleration.
Such quasi-geodesics have been considered in \cite{BatistaKrakowskiLeite2017,JurdjevicMarkinaLeite2019,KrakowskiMachadoLeiteBatista2017,MachadoLeiteBatzies2020}, where a representation of the Stiefel manifold with square matrices was used.

We introduce an economic way to compute these quasi-geodesics at considerably reduced computational costs. Furthermore, we propose a new kind of quasi-geodesics, which turn out to be closer to the true Riemannian geodesic but come at a slightly higher computational cost than the aforementioned economic quasi-geodesics. Both kinds of quasi-geodesics can be used for a wide range of problems, including optimization and interpolation of a set of points.

\section{The Stiefel Manifold}
This introductory exposition follows mainly \cite{EdelmanAriasSmith1999}.
The manifold of orthonormal frames in $\R^{n \times p}$, i.e. the \emph{Stiefel manifold}, is 
    $\St(n,p) := \left\lbrace U \in \R^{n \times p} \ \middle|\ U^TU=\I_p \right\rbrace$.
The \emph{tangent space} at any point $U \in \St(n,p)$ can be parameterized as
\begin{equation*}
    T_U\St(n,p) := \left\lbrace UA+U_\perp B \ \middle|\ A \in \so(p),\ B  \in \R^{(n-p)\times p} \right\rbrace,
\end{equation*}
where $\so(p)$ denotes the real $p\times p$ skew-symmetric matrices and $U_\perp$ denotes an arbitrary but fixed orthonormal completion of $U$ such that $\begin{pmatrix} U & U_\perp \end{pmatrix}\in \O(n)$. 
A Riemannian metric on $T_U\St(n,p)$ is induced by the \emph{canonical} inner product
\begin{equation*}
    g_U \colon T_U\St(n,p) \times T_U\St(n,p) \to \R,\ g_U(\Delta_1,\Delta_2):=\tr\big(\Delta_1^T(\I_n-\frac12 UU^T)\Delta_2\big).
\end{equation*}
The metric defines geodesics, i.e. locally shortest curves. The \emph{Riemannian exponential} gives the geodesic from a point $U \in \St(n,p)$ in direction $\Delta = UA+U_\perp B \in T_U\St(n,p)$.
Via the QR-decomposition $(\I_n-UU^T)\Delta = QR$, with $Q \in \St(n,p)$ and $R\in \R^{p \times p}$, it can be calculated as
\begin{equation*}
    \Exp_U(t\Delta):=\begin{pmatrix} U & U_\perp \end{pmatrix}\expm t\begin{pmatrix}A & -B^T\\ B & 0\end{pmatrix} \begin{pmatrix}\I_p \\ 0 \end{pmatrix}
    = \begin{pmatrix} U & Q \end{pmatrix}\expm t\begin{pmatrix}A & -R^T\\ R & 0\end{pmatrix}\begin{pmatrix}\I_p \\ 0 \end{pmatrix},
\end{equation*}
where $\expm$ denotes the matrix exponential.
The inverse problem, i.e. given $U,\tilde{U}\in \St(n,p)$, find $\Delta\in T_U\St(n,p)$ with
$\tilde{U} = \Exp_U(\Delta)$, is called the \emph{geodesic endpoint problem} and is associated with 
computing the Riemannian logarithm \cite{Zimmermann2017}. There is no known closed formula. Yet, as suggested in \cite{BatistaKrakowskiLeite2017,KrakowskiMachadoLeiteBatista2017}, one can exploit the quotient relation between $\St(n,p)$ and the 
\emph{Grassmann manifold} \cite{EdelmanAriasSmith1999,BendokatZimmermanAbsil2020} of $p$-dimensional subspaces of $\R^n$: Let $U,\tilde{U} \in \St(n,p)$. Then the columns of $U$ and $\tilde{U}$ span subspaces, i.e. points on the Grassmannian. For the Grassmannian, the Riemannian logarithm is known \cite{BendokatZimmermanAbsil2020}, which means that we know how to find $U_\perp B \in T_U\St(n,p)$ and $R \in \O(p)$ such that
    $\tilde{U}R = \begin{pmatrix} U & U_\perp\end{pmatrix}\expm\begin{psmallmatrix}0 & \smash{-B^T}\\ B & 0\end{psmallmatrix}\begin{psmallmatrix}\I_p \\ 0 \end{psmallmatrix}$.
Denote $A:=\logm(R^T)$, where $\logm$ is the principle matrix logarithm. Then
\begin{equation}
\label{eq:observation}
    \tilde{U} = \begin{pmatrix} U & U_\perp\end{pmatrix}\expm\begin{pmatrix}0 & -B^T\\ B & 0\end{pmatrix}\begin{pmatrix}\I_p \\ 0 \end{pmatrix}\expm(A).
\end{equation} 
If there is a Stiefel geodesic from $U$ to $\tilde{U}$, then there is also $U\tilde{A} + U_\perp \tilde{B} \in T_U\St(n,p)$ such that
\begin{equation}
\label{eq:observation2}
    \tilde{U} = \begin{pmatrix} U & U_\perp\end{pmatrix}\expm\begin{pmatrix}\tilde{A} & -\tilde{B}^T\\ \tilde{B} & 0\end{pmatrix}\begin{pmatrix}\I_p \\ 0 \end{pmatrix}.
\end{equation} 
Given $U,\tilde{U}$, we cannot find $U\tilde{A} + U_\perp \tilde{B} \in T_U\St(n,p)$ directly for \eqref{eq:observation2}, but we can find $UA + U_\perp B \in T_U\St(n,p)$ for \eqref{eq:observation}. On the other hand, given $U \in \St(n,p)$ and $\Delta = UA + U_\perp B \in T_U\St(n,p)$, we can define a $\tilde{U} \in \St(n,p)$ via \eqref{eq:observation}.

\section{Quasi-Geodesics on the Stiefel Manifold}
We first reduce the computational effort associated with the quasi-geodesics of
\cite{KrakowskiMachadoLeiteBatista2017,JurdjevicMarkinaLeite2019,BatistaKrakowskiLeite2017}. Then, we introduce a new technique to construct quasi-geodesics.

\subsection{Economy-size Quasi-Geodesics}
Similarly to \cite{KrakowskiMachadoLeiteBatista2017,JurdjevicMarkinaLeite2019,BatistaKrakowskiLeite2017}, we use the notion of a retraction \cite{AbsilMahonySepulchre2008} as a starting point of the construction.
Let $M$ be a smooth manifold with tangent bundle $TM$. A \emph{retraction} is a smooth mapping $R\colon TM \to M$ with the following properties:

1.) $R_x(0)=x$, i.e. $R$ maps the zero tangent vector at $x\in M$ to $x$.

2.)  The derivative at $0$, $\D R_x(0)$ is the identity mapping on $T_0T_xM \simeq T_xM$.\\
Here, $R_x$ denotes the restriction of $R$ to $T_xM$.
An example of a retraction is the Riemannian exponential mapping. On the Stiefel manifold, we can for any retraction $R\colon T\St(n,p) \to \St(n,p)$ and any tangent vector $\Delta \in T_U\St(n,p)$ define a smooth curve $\gamma_\Delta\colon t \mapsto R_U(t\Delta)$, which fulfills $\gamma_\Delta(0) = U$ and $\dot\gamma_\Delta(0)=\Delta$. The essential difference to  \cite{KrakowskiMachadoLeiteBatista2017,JurdjevicMarkinaLeite2019,BatistaKrakowskiLeite2017} is that we work mainly with $n \times p$ matrices instead of $n\times n$ representatives,
which entails a considerable cost reduction, when $p\leq \frac{n}{2}$.


The idea is to connect the subspaces spanned by the Stiefel manifold points with the associated Grassmann geodesic, while concurrently moving along the equivalence classes to start and end at the correct Stiefel representatives. This principle is visualized in \cite[Fig. 1]{BatistaKrakowskiLeite2017}. We define the economy-size quasi-geodesics similarly to \cite[Prop. 6 and Thm. 7]{KrakowskiMachadoLeiteBatista2017}.

\begin{proposition}
\label{prop:simple_retraction}
    Let $U \in \St(n,p)$ and $\Delta = UA+U_\perp B \in T_U\St(n,p)$ with compact SVD $(\I_n-UU^T)\Delta = U_\perp B \overset{\text{\tiny SVD}}{=}Q\Sigma V^T$. The mapping $\RS \colon T\St(n,p)\to \St(n,p)$, defined by $\Delta$ maps to $\RS_U(\Delta) := (UV\cos(\Sigma) + Q\sin(\Sigma))V^T\expm(A)$, is a retraction with corresponding quasi-geodesic
    \begin{equation}
    \label{eq:simple_quasigeodesic}
        \gamma(t) = \RS_U(t\Delta)=(UV\cos(t\Sigma) + Q\sin(t\Sigma))V^T\expm(tA).
    \end{equation}
    An orthogonal completion of $\gamma(t)$ is $\gamma_\perp(t)=\begin{pmatrix} U & U_\perp \end{pmatrix}\expm\left(t \begin{psmallmatrix} 0 & \smash{-B^T}\\ B & 0 \end{psmallmatrix}\right) \begin{psmallmatrix}0\\ \I_{n-p}\end{psmallmatrix}$. The quasi-geodesic $\gamma$ has the following properties:
    \begin{enumerate}
     \item $\gamma(0) = U$
     \item $\dot\gamma(t) 
     =\gamma(t)A+\gamma_\perp(t)B\expm(tA)$
     \item $\norm{\dot\gamma(t)}^2=\frac{1}{2}\tr(A^TA) + \tr(B^TB)$ (constant speed)
     \item $\ddot\gamma(t) = \gamma(t)(A^2-\expm(tA^T)B^TB\expm(tA)) + 2\gamma_\perp(t)BA\expm(tA)$
     \item $D_t\dot\gamma(t) = \gamma_\perp(t)BA\expm(tA)$
     \item $\norm{D_t\dot\gamma(t)}^2 = \norm{BA}_F^2$ (constant-norm covariant acceleration)
    \end{enumerate}
    Furthermore, $\gamma$ is a geodesic if and only if $BA=0$.
\end{proposition}
    \begin{proof}
        The fact that $\RS_U(0)=U$ for all $U \in \St(n,p)$ is obvious. Furthermore
        \begin{align*}
            \D\RS_U(0)(\Delta)
            =\ &(-UV\sin(\varepsilon\Sigma)\Sigma+Q\cos(\varepsilon\Sigma)\Sigma)V^T\exp_m(\varepsilon A) \big\vert_{\varepsilon=0}\\
            &+ (UV\cos(\varepsilon\Sigma)+ Q\sin(\varepsilon\Sigma))V^T\expm(\varepsilon A)A \big\vert_{\varepsilon=0}\\
            =\ & Q\Sigma V^T + U VV^T A = \Delta,
        \end{align*}
        so $\RS$ is a retraction. 
        Note that $\gamma$ can also be written as
        \begin{align*}
            \gamma(t)
            &=\begin{pmatrix} U & U_\perp \end{pmatrix}\expm\left(t \begin{pmatrix} 0 & -B^T\\ B & 0 \end{pmatrix}\right) \begin{pmatrix}\expm(tA)\\ 0\end{pmatrix},
        \end{align*}
        by comparison with the Grassmann geodesics in \cite[Thm. 2.3]{EdelmanAriasSmith1999}. Therefore one possible orthogonal completion is given by the stated formula. The formulas for $\dot\gamma(t)$ and $\ddot\gamma(t)$ can be calculated by taking the derivative of
        \begin{equation*}
           t\mapsto \begin{pmatrix}\gamma(t) & \gamma_\perp(t)\end{pmatrix} =\begin{pmatrix} U & U_\perp \end{pmatrix}\expm\left(t \begin{pmatrix} 0 & -B^T\\ B & 0 \end{pmatrix}\right) \begin{pmatrix}\expm(tA) & 0\\ 0 & \I_{n-p}\end{pmatrix}.
        \end{equation*}
        It follows that 
        $\norm{\dot\gamma(t)}^2
        = \tr(\dot\gamma(t)^T(\I_n-\frac{1}{2}\gamma(t)\gamma(t)^T)\dot\gamma(t))=
        \frac{1}{2}\tr(A^TA) + \tr(B^TB).$ To calculate the covariant derivative $D_t\dot\gamma(t)$, we use $\gamma(t)^T\gamma_\perp(t)=0$ and the formula for the covariant derivative of $\dot\gamma$ along $\gamma$ from \cite[eq. (2.41), (2.48)]{EdelmanAriasSmith1999},
        \begin{equation}
        \label{eq:covariant_derivative}
            D_t\dot\gamma(t)=\ddot\gamma(t)+\dot\gamma(t)\dot\gamma(t)^T\gamma(t)+\gamma(t)\left((\gamma(t)^T\dot\gamma(t))^2+\dot\gamma(t)^T\dot\gamma(t)\right),
        \end{equation}
        cf. \cite{KrakowskiMachadoLeiteBatista2017}. Since it features constant speed and constant-norm covariant acceleration, $\gamma$ is a quasi-geodesic. It becomes a true geodesic if and only if $BA=0$. \qed
    \end{proof}

Connecting $U, \tilde{U} \in \St(n,p)$ with a quasi-geodesic from Proposition \ref{prop:simple_retraction} requires the inverse of $\RS_U$. Since $\RS_U$ is the Grassmann exponential -- lifted to the Stiefel manifold -- followed by a change of basis, we can make use of the modified algorithm from \cite{BendokatZimmermanAbsil2020} for the Grassmann logarithm. The procedure is stated in Algorithm \ref{alg:simple_quasigeodesic}. Proposition \ref{prop:QG} confirms that it yields a quasi-geodesic.

 \begin{algorithm}
 \label{alg:simpleQG}
    \caption{Economy-size quasi-geodesic between two given points}
    \label{alg:simple_quasigeodesic}
    \begin{algorithmic}[1]
        \Require $U, \tilde{U} \in \St(n,p)$
        \State $\widetilde{Q} \widetilde{S} \widetilde{R}^T \overset{\text{\tiny SVD}}{:=} \tilde{U}^T U$  \Comment{SVD}
        \State $R:=\widetilde{Q}\widetilde{R}^T$
        \State $\tilde{U}_* := \tilde{U}R$ \Comment{Change of basis in the subspace spanned by $\tilde{U}$}
        \State $A := \logm(R^T)$
        \State $Q S V^T \overset{\text{\tiny SVD}}{:=}(\I_n-UU^T)\tilde{U}_*$ \Comment{compact SVD}
        \State $\Sigma := \arcsin(S)$ \Comment{element-wise on the diagonal}
        \Ensure $\gamma(t)=(UV\cos(t\Sigma) + Q\sin(t\Sigma))V^T\expm(tA)$
    \end{algorithmic}
\end{algorithm}

\begin{proposition}
\label{prop:QG}
    Let $U, \tilde{U} \in \St(n,p)$. Then Algorithm \ref{alg:simple_quasigeodesic} returns a quasi-geodesic $\gamma$ connecting $U$ and $\tilde{U}$, i.e. $\gamma(0)=U$ and $\gamma(1)= \tilde{U}$, in direction $\dot\gamma(0)=UA+Q\Sigma V^T$ and of length $L(\gamma)=(\frac{1}{2}\tr(A^TA) + \tr(\Sigma^2))^{\frac{1}{2}}$.
\end{proposition}
    \begin{proof}
      Follows from \cite[Alg. 1]{BendokatZimmermanAbsil2020}, Proposition~\ref{prop:simple_retraction} and a straightforward calculation.
    \end{proof}

\subsection{Short Economy-size Quasi-Geodesics}
To construct an alternative type of quasi-geodesics, we make the following observation: Denote $B$ in \eqref{eq:observation} by $\hat{B}$ and calculate the SVD $U_\perp \hat{B}= Q\Sigma V^T$. Furthermore, compute $R \in \O(p)$ as in Algorithm \ref{alg:simple_quasigeodesic} and denote $a:= \logm(R^T) \in \so(p)$ and $b := \Sigma V^T\in \R^{p \times p}$. Then we can rewrite \eqref{eq:observation} as
\begin{equation*}
    \begin{split}
    \tilde{U} &= \begin{pmatrix} U & Q \end{pmatrix}\expm\begin{pmatrix} 0 & -b^T\\ b & 0\end{pmatrix}\begin{pmatrix}\I_p \\ 0_{p\times p} \end{pmatrix} \expm(a)\\
    &= \begin{pmatrix} U & Q \end{pmatrix}\expm\begin{pmatrix} 0 & -b^T\\ b & 0\end{pmatrix}\expm\begin{pmatrix}a & 0\\ 0 & c\end{pmatrix}\begin{pmatrix}\I_p \\ 0 \end{pmatrix} \quad \text{ for any }c \in \so(p).
    \end{split}
\end{equation*}
Without the factor $c$, this is exactly what lead to the quasi-geodesics \eqref{eq:simple_quasigeodesic}. There are however also matrices $A \in \so(p), B\in \R^{p \times p}$ and $C \in \so(p)$ satisfying
\begin{equation}
    \begin{pmatrix} A & -B^T\\ B & C\end{pmatrix} = \logm\left(\expm\begin{pmatrix} 0 & -b^T\\ b & 0\end{pmatrix}\expm\begin{pmatrix}a & 0\\ 0 & c\end{pmatrix}\right).
\end{equation}
This implies that
\begin{equation}
    \label{eq:short_quasigeodesic}
    \rho(t) = \begin{pmatrix} U & Q \end{pmatrix}\expm\left(t\begin{pmatrix} A & -B^T\\ B & C\end{pmatrix} \right)\begin{pmatrix}\I_p \\ 0 \end{pmatrix}
\end{equation}
is a curve from $\rho(0)=U$ to $\rho(1)=\tilde{U}$. It is indeed the projection of the geodesic in $\O(n)$ from $ \begin{pmatrix} U & U_\perp \end{pmatrix}$ to $\begin{pmatrix} \tilde{U} & \tilde{U}_\perp \end{pmatrix}$ for some orthogonal completion $\tilde{U}_\perp$ of $\tilde{U}$. 
If $C=0$, then $\rho$ is exactly the Stiefel geodesic. For $x := \begin{psmallmatrix} 0 & \smash{-b^T}\\ b & 0\end{psmallmatrix}$ and $y:= \begin{psmallmatrix} a & 0\\ 0 & c\end{psmallmatrix}$
with  $\|x\| + \|y\|\leq \ln(\sqrt{2})$, we can express $C$ with help of the (Dynkin-)Baker-Campbell-Hausdorff (BCH) series $Z(x,y) = \logm(\expm(x)\expm(y))$, \cite[\S 1.3, p. 22]{rossmann2006lie}.
To get close to the Riemannian geodesics, we want to find a $c$ such that $C$ becomes small. Three facts are now helpful for the solution:
\begin{enumerate}
 \item The series $Z(x,y)$ depends only on iterated commutators $[\cdot,\cdot]$ of $x$ and $y$.
 \item Since the Grassmannian is symmetric, the Lie algebra $\so(n)$ has a Cartan decomposition 
 $\so(n)= \v \oplus \h$ with $[\v,\v] \subseteq \v,\ [\v,\h] \subseteq \h,\ [\h,\h] \subseteq \v,$ \cite[(2.38)]{EdelmanAriasSmith1999}.
 \item For $x$ and $y$ defined as above, we have $x \in \h$ and $y \in \v \subset \so(2p)$.
\end{enumerate}
Considering terms up to combined order 4 in $x$ and $y$ in the Dynkin formula and denoting the anti-commutator by $\{x,y\}= xy+yx$, the matrix $C$ is given by
\begin{align*}
 C = C(c) =\ &c + \frac1{12}\Big(2bab^T-\{bb^T,c\}\Big) - \frac{1}{24}\Big(2[c,bab^T]-[c,\{bb^T,c\}]\Big) + \text{h.o.t.}
\end{align*}
%
%
%
%
Ignoring the higher-order terms, we can consider this as a fixed point problem 
$0 = C(c) \Leftrightarrow c = -\frac1{12}(\ldots)+ \frac{1}{24}(\ldots)$.
Performing a single iteration starting from $c_0=0$ yields
\begin{equation}
    c_1 =c_1(a,b)= -\frac{1}{6}bab^T.
\end{equation}
With $\delta:= \max\{\norm{a},\norm{b}\}\leq\ln(\sqrt{2})$, where $\norm{\cdot}$ denotes the 2-norm, one can show that this choice of $c(a,b)$
produces a $C$-block with $\norm{C(c)}\leq (\frac{7}{216}+\frac{1}{1-\delta})\delta^5$.

A closer look at curves of the form \eqref{eq:short_quasigeodesic} shows the following Proposition.
\begin{proposition}
    Let $U, Q \in \St(n,p)$ with $U^T Q = 0$ and $A \in \so(p)$, $B \in \R^{p \times p}$, $C \in \so(p)$. Then the curve
    \begin{equation}
    \label{eq:short_quasigeodesic_prop}
        \rho(t) = \begin{pmatrix} U & Q \end{pmatrix}\expm\left(t\begin{pmatrix} A & -B^T\\ B & C\end{pmatrix} \right)\begin{pmatrix}\I_p \\ 0 \end{pmatrix}
    \end{equation}
    has the following properties, where $\rho_\perp(t) := \begin{pmatrix} U & Q \end{pmatrix}\expm\left(t\begin{psmallmatrix} A & -B^T\\ B & C\end{psmallmatrix} \right)\begin{psmallmatrix}0 \\ \I_p \end{psmallmatrix}$:
    \begin{enumerate}
     \item $\rho(0) = U$
     \item $\dot\rho(0)=UA + Q B \in T_U\St(n,p)$
     \item $\dot\rho(t)= \begin{pmatrix} U & Q \end{pmatrix}\expm\left(t\begin{pmatrix} A & -B^T\\ B & C\end{pmatrix} \right)\begin{pmatrix}A \\ B \end{pmatrix} = \rho(t)A+\rho_\perp(t)B \in T_{\rho(t)}\St(n,p)$
     \item $\norm{\dot\rho(t)}^2 = \frac{1}{2}\tr(A^TA) + \tr(B^TB)$ (constant speed)
     \item $\ddot\rho(t) = \begin{pmatrix} U & Q \end{pmatrix}\expm\left(t\begin{pmatrix} A & -B^T\\ B & C\end{pmatrix} \right)\begin{pmatrix}A^2-B^TB \\ BA+CB \end{pmatrix}\\
     = \rho(t)(A^2-B^TB) + \rho_\perp(t)(BA+CB)$
     \item $D_t\dot\rho(t) = \rho_\perp(t)CB$
     \item $\norm{D_t\dot\rho(t)}^2 = \norm{CB}_F^2$ (constant-norm covariant acceleration)
    \end{enumerate}
\end{proposition}
    \begin{proof}
      This can directly be checked by calculation and making use of the formula \eqref{eq:covariant_derivative} for the covariant derivative $D_t\dot\rho(t)$.
    \end{proof}

Note that the property $U^TQ=0$ can only be fulfilled if $p\leq \frac{n}{2}$. Since they feature constant speed and constant-norm covariant acceleration, curves of the form \eqref{eq:short_quasigeodesic_prop} are \emph{quasi-geodesics}.
Now we can connect two points on the Stiefel manifold with Algorithm \ref{alg:short_quasigeodesic}, making use of $\expm\begin{psmallmatrix} 0 & -\Sigma\\ \Sigma & 0\end{psmallmatrix} = \begin{psmallmatrix} \cos\Sigma & -\sin\Sigma\\ \sin\Sigma & \cos\Sigma\end{psmallmatrix}$ for diagonal $\Sigma$.
 \begin{algorithm}
    \caption{Short economy-size quasi-geodesic between two given points}
    \label{alg:short_quasigeodesic}
    \begin{algorithmic}[1]
        \Require $U, \tilde{U} \in \St(n,p)$
        \State $\widetilde{Q} \widetilde{S} \widetilde{R}^T \overset{\text{\tiny SVD}}{:=} \tilde{U}^T U$  \Comment{SVD}
        \State $R:=\widetilde{Q}\widetilde{R}^T$
        \State $a := \logm(R^T)$
        \State $Q S V^T \overset{\text{\tiny SVD}}{:=}(\I_n-UU^T)\tilde{U}R$ \Comment{compact SVD}
        \State $\Sigma := \arcsin(S)$ \Comment{element-wise on the diagonal}
        \State $b := \Sigma V^T$
        \State $c := c(a,b) = -\frac{1}{6}bab^T$
        \State $\begin{pmatrix} A & -B^T\\ B & C\end{pmatrix} = 
        \logm\begin{pmatrix} V\cos(\Sigma)V^TR^T & -V\sin(\Sigma)\expm(c)\\ \sin(\Sigma)V^TR^T & \cos(\Sigma)\expm(c)\end{pmatrix}$
        \Ensure $\rho(t) = \begin{pmatrix} U & Q \end{pmatrix}\expm\left(t\begin{pmatrix} A & -B^T\\ B & C\end{pmatrix} \right)\begin{pmatrix}\I_p \\ 0 \end{pmatrix}$
    \end{algorithmic}
\end{algorithm}
Curves produced by Algorithm \ref{alg:short_quasigeodesic} are of the form \eqref{eq:short_quasigeodesic_prop}. We numerically verified that they are closer to the Riemannian geodesic than the economy-size quasi-geodesics of Algorithm~\ref{alg:simple_quasigeodesic} in the cases we considered. As geodesics are locally shortest curves, this motivates the term \emph{short economy-size quasi-geodesics} for the curves produced by Algorithm \ref{alg:short_quasigeodesic}.

Note that Algorithm \ref{alg:short_quasigeodesic} allows to compute the quasi-geodesic $\rho(t)$ with initial velocity $\dot\rho(0) = UA+QB$ between two given points. The opposite problem, namely finding the quasi-geodesic $\rho(t)$ given a point $U \in \St(n,p)$ with tangent vector $UA+QB \in T_U\St(n,p)$, is however not solved, since the correct $C \in \so(p)$ is missing. Nevertheless, since $\rho(t)$ is an approximation of a geodesic, the Riemannian exponential can be used to generate an endpoint.

\section{Numerical Comparison}
To compare the behaviour of the economy-size and the short economy-size quasi-geodesics, two random points on the Stiefel manifold $\St(200,30)$ with a distance of $d \in \{0.1\pi, 0.5\pi, 1.3\pi\}$ from each other are created. Then the quasi-geodesics according to Algorithms \ref{alg:simple_quasigeodesic} and \ref{alg:short_quasigeodesic} are calculated. The Riemannian distance, i.e., the norm of the Riemannian logarithm, between the quasi-geodesics and the actual Riemannian geodesic is plotted in Figure \ref{fig:numerical_behaviour}.
In all cases considered, the short quasi-geodesics turn out to be two to five orders of magnitude closer to the Riemannian geodesic than the economy-size quasi-geodesics.

\begin{figure}[h]
\centering
\includegraphics[width=.8\textwidth]{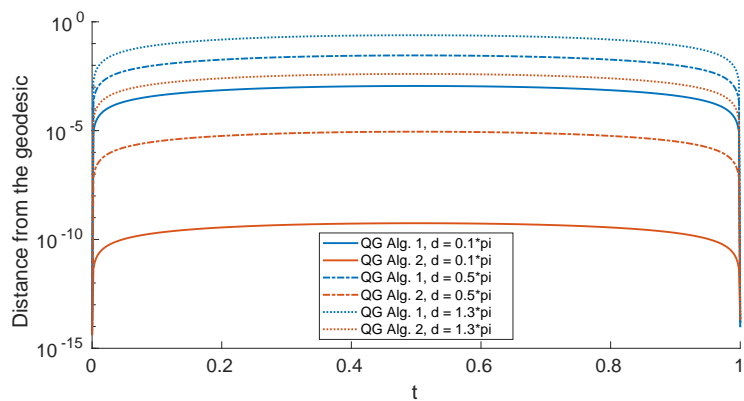}
\caption{Comparison of the distance of the quasi-geodesics (QG) to the true geodesic between two random points. The distance between the two points is denoted by $d$.}
\label{fig:numerical_behaviour}
\end{figure}

In Table \ref{tab:lengtherror}, we display the relative deviation of the length of the quasi-geodesics from the Riemannian distance between two randomly generated points at a distance of $\frac\pi2$ on $\St(200,p)$, where $p$ varies between $10$ and $100$. The outcome justifies the name \emph{short} quasi-geodesics.
\begin{table}
\centering
\caption{Comparison of the relative deviation in length of the quasi-geodesics compared to the true geodesic between two random points on $\St(200,p)$ at a distance of $\frac\pi2$ for different values of $p$. The observable $p$-dependence suggests further investigations.}
\label{tab:lengtherror}
\begin{tabular}{l@{\hskip .5cm}c@{\hskip .5cm}c}
\toprule
p & Algorithm \ref{alg:short_quasigeodesic} & Algorithm \ref{alg:simple_quasigeodesic}\\
\midrule
10 & 1.1997e-08 & 9.9308e-04\\ 20 & 3.6416e-10 & 8.6301e-04\\ 30 & 8.4590e-11 & 8.9039e-04\\ 40 & 1.9998e-11 & 7.7265e-04\\ 50 & 8.1613e-12 & 7.7672e-04\\
\bottomrule
\end{tabular}
\hskip .2cm
\begin{tabular}{l@{\hskip .5cm}c@{\hskip .5cm}c}
\toprule
p & Algorithm \ref{alg:short_quasigeodesic} & Algorithm \ref{alg:simple_quasigeodesic}\\
\midrule
60 & 4.0840e-12 & 7.2356e-04\\ 70 & 1.9134e-12 & 6.8109e-04\\ 80 & 9.9672e-13 & 6.1304e-04\\ 90 & 5.7957e-13 & 5.7272e-04\\ 100 & 2.7409e-13 & 4.9691e-04\\
\bottomrule
\end{tabular}
\end{table}

The essential difference in terms of the computational costs between the quasi-geodesics of Alg.~\ref{alg:simple_quasigeodesic}, those of Alg.~\ref{alg:short_quasigeodesic}, and
the approach in \cite[Thm. 7]{KrakowskiMachadoLeiteBatista2017} is that the they require matrix exp- and log-function evaluations of $(p\times p)$-, $(2p\times 2p)$- and
$(n\times n)$-matrices, respectively.

\section{Conclusion and Outlook}
We have proposed a new efficient representation for a well-known type of quasi-geodesics on the Stiefel manifold, which is suitable for large-scale computations and has an exact inverse to the endpoint problem for a given tangent vector. Furthermore, we have introduced a new kind of quasi-geodesics, which are much closer to the Riemannian geodesics. These can be used for endpoint problems, but the exact curve for a given tangent vector is unknown. Both kinds of quasi-geodesics can be used, e.g., for interpolation methods like De Casteljau etc. \cite{KrakowskiMachadoLeiteBatista2017,BatistaKrakowskiLeite2017}. In future work, further and more rigorous studies of the quasi-geodesics' length properties 
and the $p$-dependence displayed in Table \ref{tab:lengtherror} are of interest.
%

%
%
%
\bibliographystyle{splncs04}
\bibliography{quasigeodesicbib}
\end{document}